\documentclass[12pt]{article}

\usepackage{setup}
\bibliography{references}

\DeclareMathOperator{\Disc}{Disc}
\DeclareMathOperator{\fibpizero}{fib\pi_0}
\newcommand{\cyclic}[1]{C_{#1}}

\title{The 3-strand braid group with torsion}
\author{Ethan Dlugie, Tahsin Saffat}
\date{September 2025}

\begin{document}

\maketitle

\begin{abstract}
    In the 1950s, H. S. M. Coxeter considered the quotients of braid groups given by adding the relation that all half Dehn twist generators have some fixed, finite order. He found a remarkable formula for the order of these groups in terms of some related Platonic solids. Despite the inspiring apparent connection between these ``truncated'' braid groups and Platonic solids, Coxeter's proof boils down to a finite case check that reveals nothing about the structure present. We give a topological interpretation of the truncated 3-strand braid group that makes the connection with Platonic solids clear. One of our key tools is a formalism for orbifolds developed by A. Henriques that we think others would find interesting.
\end{abstract}

\section{Introduction}\label{sec: introduction}
The braid groups $B_n$ are a ubiquitous family of groups in modern mathematics. These groups have ``half twist'' generators $\sigma_1,\dotsc,\sigma_{n-1}$ with braid relations between adjacent generators and commuting relations between non-adjacent generators. See for instance \cite{Gonzalez-Meneses2011BasicGroups} for an introduction to braid groups. \cref{fig:braid_example} shows a basic example of a 4-strand braid.

\begin{figure}
    \centering
    
    \begin{overpic}[
        page=1,
        viewport={0.6in 9.1in 3.2in 10.6in},
        clip, grid=false, scale=0.8
    ]{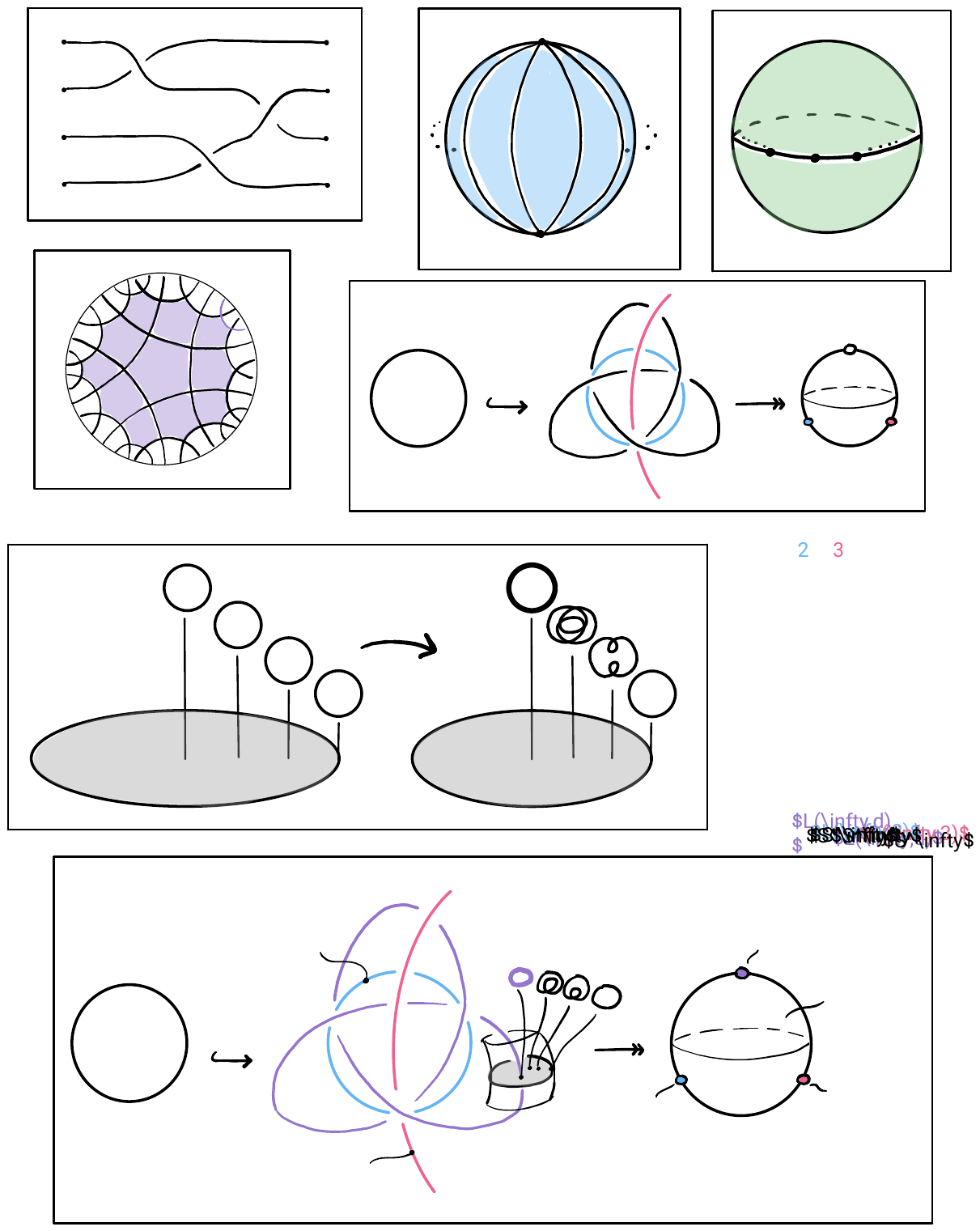}
        \put(1,50){$1$}
        \put(1,35){$2$}
        \put(1,20){$3$}
        \put(1,5){$4$}
        \put(26.5,34.5){$\sigma_1$}
        \put(49,5){$\sigma_3$}
        \put(67.5,35.5){$\sigma_2^{-1}$}
    \end{overpic}
    \caption[Braid example]{A depiction of the braid $\sigma_1\sigma_3\sigma_2^{-1} \in B_4$. The crossings considered from left-to-right match up with the word in the group read left-to-right.}
    \label{fig:braid_example}
\end{figure}
In this paper, we will be considering what we call the \textit{truncated braid groups}. Such a group, which we call $B_n(d)$, is the $n$-strand braid group with the extra relation that each generator $\sigma_i$ has a fixed finite order $d \geq 2$.
\begin{equation*}
    B_n(d) = \left\langle \sigma_1,\dotsc,\sigma_{n-1} \mid \sigma_i\sigma_{i+1}\sigma_i = \sigma_{i+1}\sigma_i\sigma_{i+1}, \sigma_i\sigma_j=\sigma_j\sigma_i \text{ if } |i-j| > 1, {\color{SquidPurple} \sigma_i^d=1} \right\rangle
\end{equation*}
These braid group quotients, and similar quotients for mapping class groups, have received attention in the last few decades in low-dimensional topology, geometric group theory, and quantum topology \cite{Humphries1992NormalGroups,Mangioni2025RigidityTwists,Dahmani2020DehnTwists,Funar2014OnUnity}.
Already over half a century ago, Coxeter considered the truncated braid groups \cite{Coxeter1959FactorGroup}. He found that $B_n(d)$ is finite in only a few exceptional cases. What is more, he discovered an incredibly remarkable formula for the order of the finite quotients.

\begin{theorem}[\cite{Coxeter1959FactorGroup}] \label{thm: Coxeter formula}
    The braid group quotient $B_n(d)$ is finite if and only if $n=2$, $d=2$ or the pair $(n,d)$ is one of the five exceptional pairs
    \begin{equation*}
        (3,3), (3,4), (3,5), (4,3), \text{ or } (5,3).
    \end{equation*}
    In such cases, the size of the quotient is given by the formula
    \begin{equation}
        | B_n(d) | = \left( \frac{f(n,d)}2 \right)^{n-1}n! \label{eq: Coxeter formula}
    \end{equation}
    where $f(n,d)$ is the number of faces in the Platonic solid made of $n$-gons with $d$ at each vertex.
\end{theorem}

See \cref{sec: Coxeter special cases} for a brief explanation of how to interpret this formula in the various cases.

Coxeter's theorem hints at a tantalizing connection between braid groups and Platonic solids. So how does he prove that his formula counts what is claimed? Well, there are only five exceptional cases of finite $B_n(d)$. Each of these is a finitely presented group. Coxeter goes one-by-one and computes the order of $B_n(d)$ either by recognizing the presentation of some known finite group or by running some coset enumeration algorithms. Then he computes the value given by the formula in \eqref{eq: Coxeter formula} and checks that the values agree in all five cases.

If the reader is shocked at this proof, they are not alone. The present authors sought to find a more natural proof that connects the geometry and topology of braid groups with the geometry and combinatorics of Platonic solids. In the present work, we report on a new proof of \cref{thm: Coxeter formula} in the $n=3$ case that is purely topological. Specifically, we find the order of a central element in $B_3(d)$.

\begin{theorem}\label{thm: order of center}
    Let $\Delta^2 = (\sigma_1\sigma_2)^3 = (\sigma_1\sigma_2\sigma_1)^2$ denote the full twist braid in $B_3$. Then the order of $\Delta^2$ in $B_3(d)$ is $\frac{f(3,d)}{2}$. This value, and so the order of the full twist, is finite if and only if $2 \leq d \leq 5$
\end{theorem}

The quotient by the central element $B_3(d) / \langle \Delta^2 \rangle$ is isomorphic to the (orientation-preserving) triangle group $\Delta(2,3,d)$, giving the short exact sequence
\begin{equation}
    1 \to \langle \Delta^2 \rangle \to B_3(d) \to \Delta(2,3,d) \to 1 \label{eq: SES central quotient}
\end{equation}
Already in the 19th century mathematicians saw the connection between the abstract triangle groups and the Platonic solids \cite{Fricke1897VorlesungenFunctionen}. It is classical that the order of $\Delta(2,3,d)$ is given by $3 \cdot f(3,d)$, and so the classical result together with \cref{thm: order of center} yields Coxeter's formula for the 3-strand braid group quotients. This is summarized later as \cref{cor: Coxeter's formula in n=3}.

Our proof of \cref{thm: order of center} is purely topological. It makes use of the trefoil knot complement (\cref{sec: braids and trefoil}), the topology of orbifolds (\cref{sec: orbispaces}), and a long exact sequence of homotopy groups for an orbifold fibration (\cref{sec: long exact sequence analysis}). We found that the standard treatments of orbifolds -- by an atlas of charts with group actions, by Lie groupoids, or by topological stacks -- were too cumbersome in various ways to allow us to use the tools we wanted to employ. In \cref{sec: orbispaces} we explain the \textit{Borel construction} employed in \cite{Henriques2001OrbispacesDefinition}, which allows us in a way to simply treat orbifolds as topological spaces. This allows us to employ all of the standard tools of algebraic topology with no alterations. We employ these tools to prove \cref{thm: order of center} in \cref{sec: long exact sequence analysis}. Finally, \cref{sec: final remarks} contains some closing thoughts about our techniques.

\section{Special cases of Coxeter's formula}\label{sec: Coxeter special cases}
We take a moment to explain how to interpret Coxeter's formula \eqref{eq: Coxeter formula} in the cases when it is not immediately obvious, i.e. in the non-sporadic finite cases and in the generic infinite case. The key is to think of a Platonic solid not as a Euclidean polyhedron but as a regular tiling of the sphere.

\begin{example}[The $n=2$ case] Here we have that the group $$B_2(d) = \langle \sigma_1 \mid \sigma_1^d \rangle$$ is just a finite cyclic group of order $d$. To interpret the quantity $f(2,d)$, we seek a tiling of the sphere by 2-gons with $d$ of them at each vertex. This is just the decomposition of the sphere by $d$-many equally spaced meridional arcs, as in \cref{fig:meridional_tiling}.

\begin{figure}
    \begin{minipage}{0.48\linewidth}
    \centering
        {\includegraphics[page=1,
            viewport={3.8in 8.6in 5.8in 10.6in},scale=0.9,
            clip]{figures.pdf}}
        \caption[Bigon tiling]{A tiling of the sphere by 2-gons cut out by meridional arcs.}
        \label{fig:meridional_tiling}
    \end{minipage}
    \hfill
    \begin{minipage}{0.48\linewidth}
    \centering
        {\includegraphics[page=1,
            viewport={6.2in 8.6in 8in 10.6in},scale=0.9,
            clip]{figures.pdf}}
        \caption[Equatorial tiling]{A tiling of the sphere by two congruent spherical polygons meeting at the equator.}
        \label{fig:equatorial_tiling}
    \end{minipage}
\end{figure}

There are $d$-many 2-gons in this tiling, so Coxeter's formula gives
\begin{equation*}
    \left( \frac{f(n,d)}2 \right)^{n-1}n! = \left( \frac{(d)}{2} \right)^{(2)-1} 2! = d,
\end{equation*}
which is indeed the order of the finite cyclic group $B_2(d)$.
\end{example}

\begin{example}[The $d=2$ case] Group theorists will quickly recognize from the presentation that $B_n(2)$ is isomorphic to the symmetric group on $n$-elements. The quotient map $B_n \to B_n(2)$ is exactly the canonical map that sends a braid to the permutation it induces on its strands.

For Coxeter's formula, we seek a tiling of the sphere by $n$-gons with 2 at each vertex. The picture here is a simple splitting of the sphere by a great circle, with $n$ equally spaced vertices along the great circle, as depicted in \cref{fig:equatorial_tiling}. Now $f(n,2)=2$ and Coxeter's formula yields
\begin{equation*}
    \left( \frac{f(n,d)}2 \right)^{n-1}n! = \left( \frac{(2)}2 \right)^{n-1} n! = n!,
\end{equation*}
which is famously the order of the symmetric group.

\end{example}

\begin{example}[An infinite case] Consider the group $B_5(4)$. To compute the number $f(5,4)$, we seek a tiling composed of regular pentagons with 4 of them at each vertex. This is impossible on the sphere, but it is possible as a tiling of the hyperbolic plane as in \cref{fig:hyperbolic_tiling}.

\begin{figure}
    \centering
    {\includegraphics[page=1,
        viewport={0.6in 6.8in 2.6in 8.6in},
        clip]{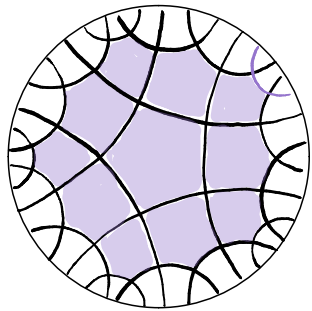}}
    \caption[Hyperbolic tiling]{A portion of the regular tiling of the hyperbolic plane by pentagons with 4 pentagons at each vertex.}
    \label{fig:hyperbolic_tiling}
\end{figure}

Evidently then $f(5,4) = \infty$, which agrees with Coxeter's assertion that $B_5(4)$ is an infinite group.
\end{example}

It becomes clear now that all of the cases when $B_n(d)$ is infinite correspond to the regular tilings of the Euclidean or hyperbolic planes. This point of view even works for the $d=\infty$ case (i.e. the usual braid group) for which one would take a tiling of the hyperbolic plane by regular \textit{ideal} $n$-gons.

\section{The 3-strand braid group and the trefoil knot}\label{sec: braids and trefoil}
In this section we explain the connection between the 3-strand braid group and the trefoil knot, a classical story. We will upgrade this to a connection between the truncated braid group $B_3(d)$ and an orbifold in the next section. We collect the necessary facts in a sequence of propositions.

\begin{proposition}\label{prop: braid group is trefoil complement}
    The 3-strand braid group is isomorphic to the fundamental group of the trefoil knot complement.
\end{proposition}
\begin{proof}
    There are several ways to see this result; we take the route of configuration spaces. Recall that the braid group is the fundamental group of the configuration space of $n$ distinct, unlabeled points in the plane. We think of these points as complex numbers $z_1,\dotsc,z_n \in \mathbb C$. By translating and dilating a collection of such points, we find that a homotopy equivalent configuration space is that for which the points have center of mass at the origin and which satisfy a chosen equation on their norms. For our $n=3$ case then, we will recognize the 3-strand braid group as the fundamental group of the space
    \begin{equation*}
        \left\{ \{z_1,z_2,z_3\} \subset \mathbb C \mid z_i \ne z_j, z_1+z_2+z_3=0, |z_1z_2+z_2z_3+z_3z_1|^2 + |z_1z_2z_3|^2 = 1 \right\}.
    \end{equation*}
    The last condition may seem strange, but it becomes more natural if we think of the three complex numbers as the roots of a monic cubic polynomial. Then our space is
    \begin{equation*}
        \{ z^3 + az+b \mid \Disc_z(z^3+az+b) \ne 0, |a|^2+|b|^2 = 1\}.
    \end{equation*}
    where $\Disc_z(p(z))$ is the discriminant of a polynomial $p$. In this case, this finally lets us see the space as
    \begin{equation}
        \{(a,b) \in \mathbb C^2 \mid |a|^2+|b|^2 = 1, -4a^3-27b^2 \ne 0 \} \subset S^3. \label{eq:braid group variety complement}
    \end{equation}
    The space here is the subset of the unit 3-sphere $S^3 \subset \mathbb C^2$ that lies outside of some polynomial variety. The particular polynomial here cuts out a $(2,3)$-torus knot, better known as the trefoil.
\end{proof}

To go along with later notation, we give the following name to the space encountered in this proof.

\begin{definition}\label{def:total space underlying spaces}
    Write $E_0 = \{(a,b) \in C^2 \mid |a|^2+|b|^2=1, -4a^3-27b^2 \ne 0\}$ for the subset unit 3-sphere in $\mathbb C^2$ away from the discriminant locus.
\end{definition}
\begin{remark}
    We would also write $E \subset \mathbb C$ for the entire 3-sphere, but we will construct $E$ later as a space glued from $E_0$ and a solid torus. We will recognize this glued space as the 3-sphere in \cref{lemma: E is the 3-sphere}.
\end{remark}

\begin{proposition}\label{prop: meridians are generators}
    Any meridional loop around the trefoil knot represents the conjugacy class of a $\sigma_i$ half-twist generator of $B_3$.
\end{proposition}
\begin{proof}
    All meridional loops around the knot are freely homotopic (up to orientation), and so they represent the same conjugacy class in the fundamental group. One such loop can be taken as the family of polynomials $p_t(z) = (z+1+1\exp(i\pi t))(z+1-1\exp(i\pi t))(z-1)$ (with the roots rescaled so that the coefficients satisfy the norm condition in \eqref{eq:braid group variety complement}). See \cref{fig:trefoil_meridian} to see that this loop is indeed a meridian of the trefoil knot. This family of polynomials represents the braid generator $\sigma_2$ that swaps two points with a right-handed half twist, which gives the proposition.

    \begin{figure}
        \centering
        \includegraphics[width=0.4\linewidth]{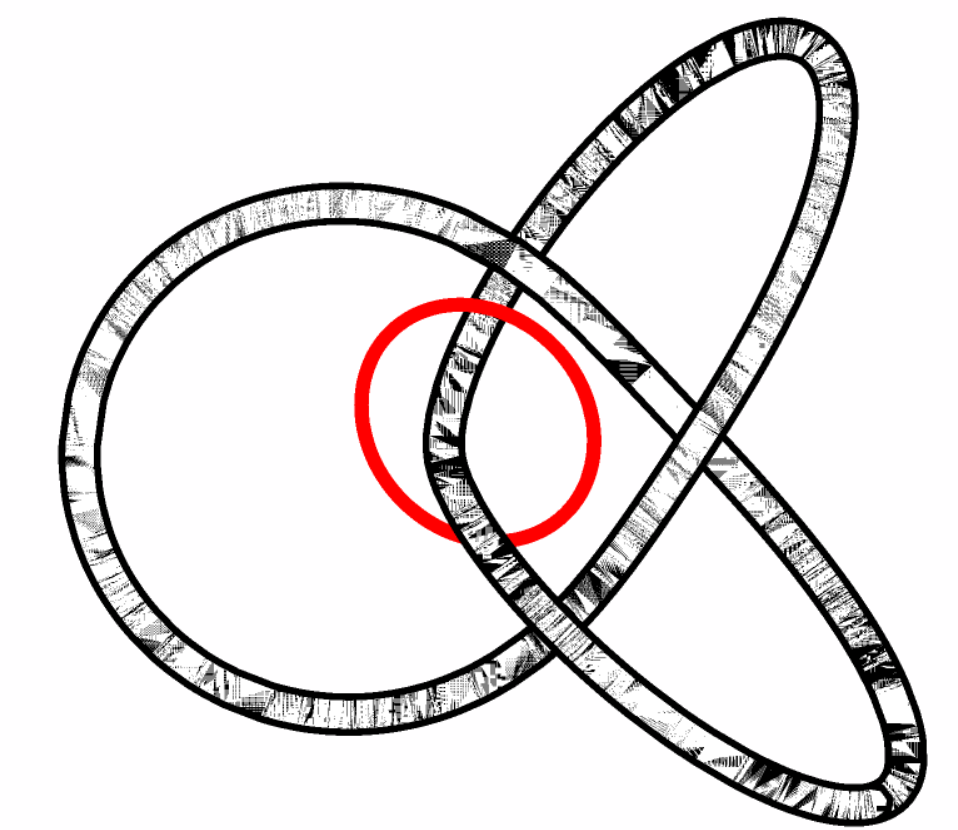}
        \caption{A computer-generated image of the discriminant locus (black/white) in $S^3 \subset \mathbb C^2$ and the loop (red) traced out by a family of polynomials whose roots swap with a half Dehn twist. We see that the discriminant locus cuts out a trefoil knot in the 3-sphere and the polynomial family traces out a meridional loop.}
        \label{fig:trefoil_meridian}
    \end{figure}
\end{proof}

\begin{proposition}\label{prop: trefoil Seifert structure}
    The trefoil complement has an oriented Seifert fiber structure. The loop traced by the generic orbit represents the central twist $\Delta^2 \in B_3$.
\end{proposition}
\begin{proof}
    An oriented Seifert fiber structure on a 3-manifold is equivalent to an action on the manifold by the circle group with finite point stabilizers. Here we consider $S^1 \subset \mathbb C$ as the unit complex numbers and use the action that multiplies a collection of points in $\mathbb C$ by an element of $S^1$. The loop traced through configuration space by a generic orbit here clearly represents the full twist braid $\Delta^2 \in B_3$. Considering how the action affects the coefficients of a degree 3 polynomial, we have
    \begin{equation}
        \zeta .(a,b) = (\zeta^2 a,\zeta^3 b). \label{eqn:circle action on 3-sphere}
    \end{equation}
    The reader can verify that this action preserves the discriminant condition of \eqref{eq:braid group variety complement}, i.e. we have a well-defined action on the trefoil complement $E_0$. We think of this as a weighted version of the circle action that gives the Hopf fibration.
    
    The only points in the 3-sphere that have non-trivial stabilizer under this action are points of the form $(a,0)$ with stabilizer $\{ -1, 1\} \approx \cyclic{2}$, and points of the form $(0,b)$ with stabilizer $\{1, \exp(2\pi i/3), \exp(4\pi i/3) \} \approx \cyclic{3}$.
\end{proof}

We record the definition of the circle action for future reference.
\begin{definition}\label{def:circle action on E}
    Define the action $S^1 \curvearrowright S^3$ by $\zeta \cdot (a,b) = (\zeta^2 a, \zeta^3 b)$. This also restricts to an action $S^1 \curvearrowright E_0$.
\end{definition}

\begin{proposition}\label{prop: trefoil base orbifold}
    The base orbifold of the Seifert fibration is the $(2,3,\infty)$ triangle orbifold.
\end{proposition}
\begin{proof}
    The ``base orbifold'' of a Seifert fibration is the quotient orbifold of the circle action, with orbifold points corresponding to the exceptional orbits having nontrivial stabilizer. Here we have pairs of complex numbers considered up to a weighted complex multiple. The map $(a,b) \mapsto [a^3:b^2]$ gives a well-defined, surjective, continuous map from $S^3$ to $\mathbb{CP}^1$ whose fibers are exactly the orbits of the $S^1$ action. In other words, the quotient space of the circle action is topologically $\mathbb{CP}^1$. But now, we see $\mathbb{CP}^1$ with one orbifold point of order 2 at the point $[1:0]$, one orbifold point of order 3 at the point $[0:1]$, and a puncture at the point $[27:-4]$ corresponding to the trefoil knot. The situation is depicted in \cref{fig:Seifert_fiber_bundle}.
\end{proof}

Note that the exact sequence arising from this Seifert fiber structure is the short exact sequence \eqref{eq: SES central quotient}.




\begin{figure}
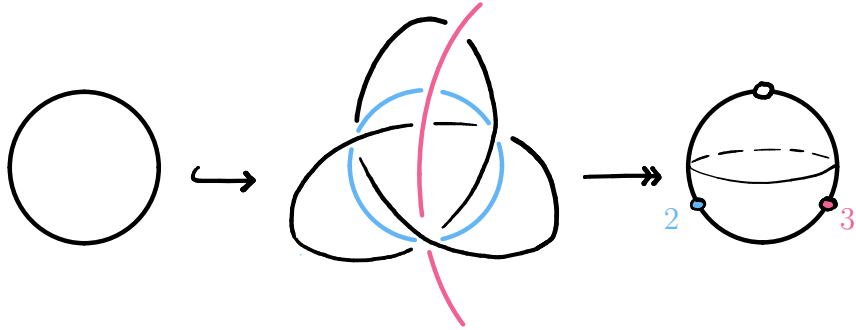

    \centering

    \begin{overpic}[
        page=1,
        viewport={3.2in 6.6in 7.8in 8.4in},
        clip, grid=false
    ]{figures.pdf}
        \put(77,12){\color{SquidBlue}$2$}
        \put(97,12){\color{SquidRed}$3$}

    \end{overpic}
    \caption[Seifert fiber bundle]{The Seifert fiber bundle that we will utilize to connect the group $B_3(d)$ with regular 2-dimensional tilings.}
    \label{fig:Seifert_fiber_bundle}
\end{figure}

\section{Orbispaces} \label{sec: orbispaces}
Our goal is to fill in the trefoil knot with an orbifold locus of order $d$ and thereby obtain the group $B_3(d)$ as an orbifold fundamental group. The Seifert fibration depicted in \cref{fig:Seifert_fiber_bundle} should fill in to give a fiber bundle in the category of orbifolds, and we want to use the associated long exact sequence of homotopy groups to analyze the homomorphism $\Delta^2 \to B_3(d)$. This makes sense intuitively, but we run into problems making it formal (e.g. how does one define ``higher homotopy groups'' of orbifolds in a rigorous but not-too-complicated way?).

To deal with these problems, we utilize an overlooked construction of Andr\'e Henriques that gives a definition of ``orbispace'' that is both intuitive and easy to work with on a formal level \cite{Henriques2001OrbispacesDefinition}. Roughly, Henriques defines an orbispace $M$ to be a continuous map of topological spaces $p:PM \to QM$ which looks something like a fibration with $K(G,1)$ fibers. These are locally modeled on the projection $p:(X \times EG)/G \to X/G$ where $G \curvearrowright X$ is a ``nice'' action of a discrete group and $EG$ is a contractible space with free $G$-action. The base space is the intuitive picture of an almost-free quotient while the ``Borel space'' above, which has no singularities, captures the algebraic topology of the quotient. We encourage the reader to read Henriques's paper, which the authors find is quite approachable and has good examples.

We start by defining all of the spaces we will need. We make extensive use of the infinite dimensional unit sphere in Hilbert space
\begin{equation*}
    S^\infty = \left\{\vec z = (z_1,z_2,\dotsc) \in \ell^2(\mathbb C) \mid \sum |z_i|^2 = 1\right\}.
\end{equation*}
The $\infty$-sphere is contractible and has free, proper actions by the circle group $S^1$ and all finite cyclic groups $\cyclic{n}$ (by viewing these as groups of complex roots of unity) via complex multiplication.

For any group action $G \curvearrowright X$, we will write the elements of the quotient $X/G$ as $G \cdot x$ or $Gx$ for $x \in X$.

\begin{definition}\label{def: orbidisk}
    Write $D^2$ for the unit disk in the complex plane and $D^2_\circ = D^2 \setminus\{0\}$ for the disk punctured at the origin. Define the \emph{order-$d$ orbidisk} to be the space $$\mathbb D^2(d) = (D^2 \times S^\infty)/\cyclic{d}$$ where the action $\cyclic{d} \curvearrowright D^2 \times S^\infty$ is simultaneous multiplication by complex roots of unity.
    
    Define the \emph{punctured orbidisk} to be the subspace $\mathbb D^2_0(d) = (D^2_0 \times S^\infty) / \cyclic d$.

    The orbidisk has a canonical projection $p_D:\mathbb D^2(d) \to D^2/\cyclic{d}$ by just forgetting the $S^\infty$ factor.

    There are free actions by the circle group $S^1 \curvearrowright \mathbb D^2(d)$ and $S^1 \curvearrowright D^2/\cyclic{d}$ by
    \begin{align*}
        \zeta \cdot (\cyclic{d}(p,\vec z)) & = \cyclic{d}(\zeta^{1/d}p,\zeta^{1/d} \vec z) && \text{for } p \in D^2, \vec z \in S^\infty, \zeta \in S^1, \text{and}\\
        \zeta \cdot (\cyclic{d} (p)) & = \cyclic{d}(\zeta^{1/d} p ) && \text{for } p \in D^2, \zeta \in S^1.
    \end{align*}
    Note that this is well-defined since the $d$-th roots of $\zeta$ differ by elements of $\cyclic{d}$. Note also that $p_D$ is now a $S^1$-equivariant map.
\end{definition}

We frequently switch back and forth between a disk $D^2$ and its quotient $D^2/\cyclic{d}$, which are homeomorphic spaces. Away from the origin, their associated bundle spaces are also the same. We record all of this here.

\begin{definition}\label{def:cyclic quotient identifications}
    Define $\psi:D^2 \to D^2/\cyclic{d}$ by
    \begin{equation*}
        \psi(p) = \cyclic{d} \cdot |p|\left( \frac p{|p|} \right)^{1/d}
    \end{equation*}
    with the limiting value $\psi(0) = \cyclic{d} \cdot 0$. This is a homeomorphism with continuous inverse
    \begin{equation*}
        \psi^{-1}(\cyclic{d} \cdot p) = |p| \left( \frac p{|p|} \right)^d
    \end{equation*}
    and the obvious modification for $p=0$.

    Define $\Psi:D^2_0 \times S^\infty \to \mathbb D^2_0(d) $ by
    \begin{equation*}
        \Psi(p,\vec z) = \cyclic{d}\left( |p|\left( \frac p{|p|} \right)^{1/d}, \left( \frac p{|p|} \right)^{1/d} \vec z \right).
    \end{equation*}
    Note that here one needs $p \ne 0$ so that $p/|p|$ (the ``angular part'' of $p$) is defined. This map is a homeomorphism with continuous inverse
    \begin{equation*}
        \Psi^{-1}(\cyclic{d}(p,\vec z)) = \left( \psi^{-1}(\cyclic{d} \cdot p),\left( \frac p{|p|} \right)^{-1} \vec z \right).
    \end{equation*}
\end{definition}


The following lemma is a direct consequence of our various definitions.
\begin{lemma}\label{lemma:disk identifications equivariant}
    The following diagram commutes and is $S^1$-equivariant:

    \begin{center}
    \begin{tikzcd}
        \mathbb D^2(d)  \arrow[d,"p_D"] & D^2_0 \times S^\infty \arrow[l,"\Psi"] \arrow[d]\\
        D^2/\cyclic{d} & D^2_0 \arrow[l,"\psi"]
    \end{tikzcd}
    \end{center}

    Here the map at right is the projection onto the first factor. The $S^1$-actions on the two spaces at left are as in \cref{def: orbidisk} and the actions at right are the multiplication by complex scalars.
\end{lemma}

One consequence of these constructions is that we can view $p_D$ as a map from the orbidisk to the disk $D^2$, rather than to $D^2/\cyclic{d}$. Away from the origin, the map is a trivial $S^\infty$ bundle. Over the origin, the fiber of $p_D$ is a copy of the infinite Lens space $L(\infty,d) \approx K(\cyclic{d},1)$. See \cref{fig:orbidisk} for a conceptual cartoon of this structure.

\begin{figure}
    \centering
    {\includegraphics[width=0.8\linewidth,page=1,
        viewport={0.4in 4in 6in 6.2in},
        clip]{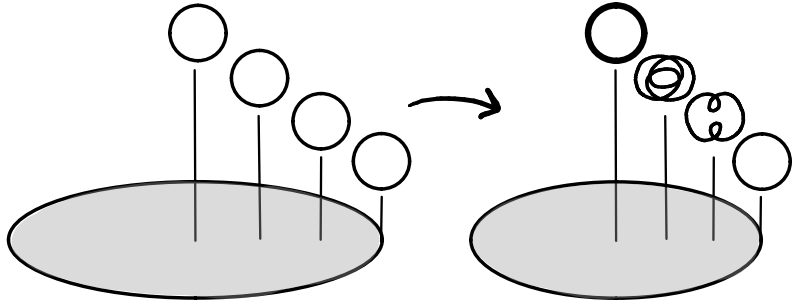}}
    \caption[Orbidisk]{The orbispace picture for $\mathbb D^2(d)$. At left is the product $D^2 \times S^\infty$, with $S^\infty$ represented by a circle. At right is the quotient by the $\cyclic{d}$ action: a disk with $S^\infty$ fibers lying over every non-origin point, mapping $d$-to-$1$ to the Lens space fiber over the origin.}
    \label{fig:orbidisk}
\end{figure}




Now we construct the orbispace that models the truncated braid group. First we need a way to glue our spaces to $E$.
\begin{definition}
    Let $\nu:D^2 \hookrightarrow E$ be a (smooth) embedding of the disk into $E \approx S^3$ which is transverse to the orbits of the $S^1$-action of \cref{def:circle action on E}, intersects each such orbit at most once, and such that $\nu(0)$ satisfies $-4a^3-27b^2=0$.\footnote{We think of $\nu$ as a normal disk to the trefoil knot in the 3-sphere.}

    Furthermore, define the embedding $N:D^2_0 \times S^\infty \times S^1 \to E_0 \times S^\infty$ by
    \begin{equation*}
        N(p,\vec z, \zeta) = (\zeta \cdot \nu(\zeta^{-1}p),\vec z)
    \end{equation*}
\end{definition}

Roughly, the map $N$ projects a point in the punctured solid torus $D^2_0 \times S^1$ to the disk $D^2_0$, maps that point into $E_0$ by $\nu$, and then pushes the point along the $S^1 \curvearrowright E_0$ action. The $S^\infty$ factor is unaffected.

The following again is a quick consequence of the definitions.
\begin{lemma}\label{lemma:disk embeddings equivariant}
    The following diagram commutes and is $S^1$-equivariant:

    \begin{center}
    \begin{tikzcd}
        D^2_0 \times S^\infty \times S^1 \arrow[r,"N"] \arrow[d] & E_0 \times S^\infty \arrow[d]\\
        D^2_0 \times S^1
        \arrow[r,"{(p,\zeta) \mapsto \zeta \cdot \nu(p)}"]
        & E_0
    \end{tikzcd}
    \end{center}

    Here the vertical maps are the obvious projections. The $S^1$-actions on the two spaces at right are as in \cref{def:circle action on E} and the actions at left are the multiplication by complex scalars.
\end{lemma}


If we cross the diagram of \cref{lemma:disk identifications equivariant} with $S^1$ and adjust the group actions accordingly, then the diagrams past together into one large $S^1$-equivariant diagram. The consequence of all of this is that we can now glue all of our spaces together into one larger space with a well-behaved $S^1$-action.

\begin{definition}\label{def: orbispace for braid group}
    Define the \textit{total space}
    \begin{equation*}
        \mathbb E = (\mathbb D^2(d) \times S^1) \bigcup_{D^2_0 \times S^\infty \times S^1} (E_0  \times S^\infty)
    \end{equation*}
    to be the pushout of the top rows of the diagrams of \cref{lemma:disk identifications equivariant,lemma:disk embeddings equivariant}.

    Similarly, define
    \begin{equation*}
        E = (D^2/\cyclic{d} \times S^1) \bigcup_{D^2_0 \times S^1} (E_0)
    \end{equation*}
    to be the pushout of the bottom rows. Let $p_E:\mathbb E \to E$ be the induced, $S^1$-equivariant continuous map.
\end{definition}


We now have the space that will serve as our model for the truncated braid group (see \cref{prop: truncated braid group is pi1}). The $S^1$ action it carries mirrors the Seifert fiber structure of \cref{prop: trefoil Seifert structure}, so we take the quotient space as in \cref{prop: trefoil base orbifold}.

\begin{definition}\label{def: orbispace base space}
    Define the \textit{base space} $\mathbb B$ as the quotient $\mathbb E / S^1$. Define $B = E/S^1$. Finally, let $p_B:\mathbb B \to B$ be the induced map.
\end{definition}

To make the connection with the constructions in \cref{sec: braids and trefoil} even clearer, we make note of the following identifications.
\begin{lemma}\label{lemma: E is the 3-sphere}
    There is an $S^1$-equivariant $E \xrightarrow{\approx} S^3$ and a homeomorphism $B \xrightarrow{\approx} \mathbb{CP}^1$ that make the following diagram commute

    \begin{center}
        \begin{tikzcd}
            E \arrow[rr] \arrow[d, "\approx"]                 & & B \arrow[d, "\approx"] \\
            S^3 \arrow[rr, "{(a,b) \mapsto [a^3:b^2]}"] & & \mathbb{CP}^1   
        \end{tikzcd}
    \end{center}
\end{lemma}
\begin{proof}
    First, we have the obvious inclusion $E_0 \subset S^3$. We also take a map $(D^2/\cyclic{d}) \times S^1 \to S^3$ by $(\cyclic{d} \cdot p, \zeta) \mapsto \zeta \cdot \nu(\zeta^{-1}\psi^{-1}(\cyclic{d} \cdot p))$. This gives an $S^1$-equivariant commutative diagram
    \begin{center}
        \begin{tikzcd}
            D^2/\cyclic{d} \times S^1 \arrow[rd, bend right] & D^2_0 \times S^1 \arrow[r] \arrow[l] &  E_0 \arrow[ld, bend left] \\[-5pt]
                                   & S^3                                    &                                                
        \end{tikzcd}
    \end{center}
    The reader can verify that $S^3$ actually satisfies the universal property for $E$, giving a canonical $S^1$-equivariant homeomorphism $E \approx S^3$. So then the quotient space $E/S^1$ is canonically homeomorphic to the quotient of $S^3$ by the $S^1$-action of \eqref{eqn:circle action on 3-sphere}. But this quotient is $\mathbb{CP}^1$ as we saw in the proof of \cref{prop: trefoil base orbifold}.
\end{proof}

Later we will seek to understand the universal cover of $\mathbb B$. For this, it will be useful to understand the structure of the map $p_B$.

\begin{proposition}\label{prop: structure of the base space projection}
    The map $p_B:\mathbb B \to B \approx \mathbb{CP}^1$ is isomorphic to 
    \begin{itemize}
        \item the projection $\mathbb{D}^2(2) \to D^2/\cyclic 2$ over a neighborhood of $[1:0]$, 
        \item the projection $\mathbb{D}^2(3) \to D^2/\cyclic 3$ over a neighborhood of $[0:1]$,
        \item the projection $\mathbb{D}^2(d) \to D^2/\cyclic d$ over a neighborhood of $[27:-4]$, or
        \item an $S^\infty$ bundle away from these three points.
    \end{itemize}
\end{proposition}
\begin{proof}
    For a point $x \in \mathbb{CP}^1$, let $\tilde x$ be a single point in the preimage in $E \approx S^3$. Let $\tilde D \subset S^3$ be a disc containing $\tilde x$ in its interior that is transverse to the orbits of the $S^1 \curvearrowright S^3$ action, is invariant under the stabilizer subroup $\stab(\tilde x) \leq S^1$, and intersects at most one of the preimages of the three points under consideration here or intersects none of those preimages if $x$ is none of the three points. Then $D \coloneqq S^1 \cdot \tilde D$ is an open neighborhood of $x \in \mathbb{CP}^1$ and $p_B^{-1}(D)$ is the projection to $\mathbb B$ of $p_E^{-1}(\tilde D)$. Here is the relevant diagram for convenience.

    \begin{center}
    \begin{tikzcd}
        \mathbb E \arrow[d,"p_E"] \arrow[rrr,"/S^1"] &[-25pt]&&[-25pt] \mathbb B \arrow[d,"p_B"] \\
        E \arrow[r,phantom,"\approx"] & S^3 \arrow[r,"/S^1"]                 & \mathbb{CP}^1 \arrow[r,phantom,"\approx"] & B
    \end{tikzcd}
    \end{center}

    Whenever $x \ne [27:-4]$, we can take $\tilde D$ to lie in $E_0$ so that $p_E^{-1}(\tilde D) = \tilde D \times S^\infty$. Then the map $p_B:p_B^{-1}(D) \to D$ is just the map $(\tilde D \times S^\infty)/\stab(\tilde x) \to \tilde D/\stab(\tilde x)$. 
    \begin{itemize}
        \item If $x=[1:0]$, then $\stab(\tilde x) = \cyclic 2$. This gives $\mathbb D^2(2) \to D^2/\cyclic 2$.
        \item If $x=[0:1]$, then $\stab(\tilde x) = \cyclic 3$. This gives $\mathbb D^2(3) \to D^2/\cyclic 3$.
        \item If $x \ne [27:-4]$ and is also neither of these two points, then $\stab(\tilde x) = \{1\}$. This gives the trivial product bundle $\tilde D \times S^\infty \to \tilde D$.
    \end{itemize}

    Finally if $x=[27:-4]$, then $\stab(\tilde x) = \{1\}$. We can take $\tilde D$ to be the disk $\nu(D^2)$ as used in the construction of $\mathbb E$ above. Now $p_E^{-1}(\tilde D) = \mathbb D^2(d) \times S^1$, whose image in $\mathbb B$ is $\mathbb D^2(d)$. The neighborhood $D = \tilde D \subset \mathbb{CP}^1$ is homeomorphic to $D^2/\cyclic d$. We see the desired structure.
\end{proof}

If our model of orbifold was that of an atlas of charts with group actions, then the map $E \to B$ would not be a fiber bundle in the usual sense. The orbifold structure allows for fibers to degenerate over singular points or for some fibers to be made entirely of singular points (both of which happen in our case). With the orbispace viewpoint presented here, we have an honest fiber bundle to which we will apply the long exact sequence of homotopy groups.

\begin{figure}
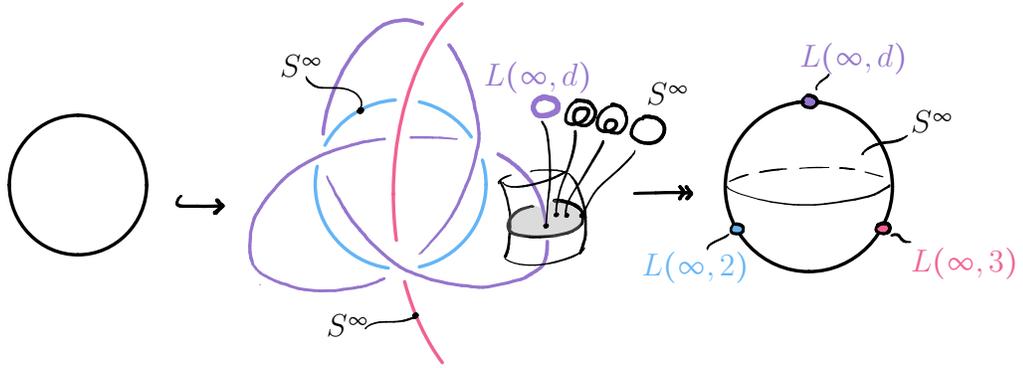

    \centering

    \begin{overpic}[
            width=0.9\linewidth, page=1,
            viewport={0.8in 0.8in 7.2in 3.6in},
            clip, grid=false
        ]{figures.pdf}
            \put(30,33){$S^\infty$}
            \put(35,5){$S^\infty$}
            \put(52,32){\color{SquidPurple} $L(\infty,d)$}
            \put(69.5,30){$S^\infty$}
            \put(69,11.5){\color{SquidBlue} $L(\infty,2)$}
            \put(98,12){\color{SquidRed} $L(\infty,3)$}
            \put(86,34){\color{SquidPurple}$L(\infty,d)$}
            \put(98,27){$S^\infty$}
        
    \end{overpic}
    \caption[The orbispace fiber bundle]{The orbispace version of the ``orbifold fiber bundle'' of \cref{fig:Seifert_fiber_bundle}. We draw the orbispaces as ``fiber bundles'' over manifolds with fibers over various points labeled.}
    \label{fig:orbispace_bundle}
\end{figure}

\section{Analysis of the long exact sequence} \label{sec: long exact sequence analysis}
The starting point of our argument is the following lemma.
\begin{lemma}\label{lemma: we have a fibration}
    The projection map $\mathbb E \to \mathbb B$ is a fiber bundle with fiber $S^1$.
\end{lemma}
\begin{proof}
    This is simply because the projection map is the quotient by the free action of a compact Lie group on a completely regular topological space. See \cite[Theorem 3.6]{Gleason1950SpacesTransformations} for a primary reference.
\end{proof}

This lemma along with our understanding of the projection maps from \cref{prop: structure of the base space projection} lets us draw the fiber bundle as in \cref{fig:orbispace_bundle}.

A standard result in algebraic topology now gives us an associated long exact sequence (see e.g. \cite[Theorem 4.41]{Hatcher2002AlgebraicTopology}).

\begin{corollary}\label{cor: long exact sequence}
    There is a long exact sequence of homotopy groups
    \begin{equation}\label{eqn: long exact sequence}
        \dotsb \to \pi_2(\mathbb B) \xrightarrow{\partial} \pi_1(S^1) \to \pi_1(\mathbb E) \to \pi_1(\mathbb B) \to \pi_0(S^1) \to \dotsb
    \end{equation}
\end{corollary}

Our main result \cref{thm: order of center} follows from a computation of the boundary map $\partial:\pi_2(\mathbb B) \to \pi_1(S^1)$. The rest of this section is devoted to computing this map as well as the other terms shown in this exact sequence.

\begin{proposition}\label{prop: truncated braid group is pi1}
    There is an isomorphism $\pi_1(\mathbb E) \approx B_3(d)$. The homomorphism $\pi_1(S^1) \to \pi_1(\mathbb E)$ takes the generator $1 \in \mathbb Z$ to  $\Delta^2 \in B_3(d)$.
\end{proposition}
\begin{proof}
    Note that $$\pi_1(\mathbb D^2(d) \times S^1) = \pi_1(\mathbb D^2(d)) \times \pi_1(S^1) = \cyclic d \times \mathbb Z.$$ The identification of $\pi_1(\mathbb D^2(d))$ follows because $\mathbb D^2(d)$ is the quotient of a contractible space by a covering space action of the cyclic group $\cyclic{d}$.

    Also we have $$\pi_1(E_0 \times S^\infty) = \pi_1(E_0) \times \pi_1(S^\infty) = B_3 \times \{1\}$$
    as was shown in \cref{prop: braid group is trefoil complement}.

These two spaces are glued together along a subspace homeomorphic to $D^2_0 \times S^\infty \times S^1$ to construct $\mathbb E$. The Seifert-Van Kampen theorem allows us to compute the new fundamental group (as in e.g. \cite{Hatcher2002AlgebraicTopology}). In our case, the $D^2_0$ factor identifies the $\cyclic d$ generator with the meridian loop in $E_0$, i.e. a half twist generator $\sigma_i \in B_3$ as per \cref{prop: meridians are generators}. The $S^1$ factor identifies the generator of $\mathbb Z$ with the orbit of the $S^1$ action in $E_0$, i.e. the full twist $\Delta^2 \in B_3$ as per \cref{prop: trefoil Seifert structure}. The resulting fundamental group is $B_3(d)$.
\end{proof}

To continue, we recall a construction of Henriques.
\begin{definition}[\cite{Henriques2001OrbispacesDefinition}]
    For a continuous map $p:X \to Y$, the \textit{fiberwise-$\pi_0$} is defined to be the space
    \begin{equation*}
        \fibpizero(X \xrightarrow{p} Y) = X/\sim
    \end{equation*}
    where $x_1 \sim x_2$ if and only if $p(x_1)=p(x_2)=y \in Y$ and $x_1$ and $x_2$ lie in the same path component of $p^{-1}(y)$.
\end{definition}

The fiberwise-$\pi_0$ gives a nice formalization of the local branched cover of an orbifold. We use this perspective in the next proposition.

\begin{proposition}\label{prop: order of triangle group}
    The order of $\pi_1(\mathbb B)$ is given by $3\cdot f(3,d)$.
\end{proposition}
\begin{proof}
    This is a classical result (with the right setup), but we give a proof again here for completeness. The strategy is to compute the degree of the universal cover $\widetilde{\mathbb B}$ of $\mathbb B$. This will come down to considering the following commutative diagram.
    
    \begin{center}
        \begin{tikzcd}
            \widetilde{\mathbb B} \arrow[r] \arrow[d] & \mathbb B \arrow[d,"p_B"] &[-25pt]\\
            \fibpizero(\widetilde{\mathbb B} \to B) \arrow[r] & B \arrow[r,phantom,"\approx"] & \mathbb{CP}^1
        \end{tikzcd}
    \end{center}

    As per \cref{prop: structure of the base space projection}, the map $p_B$ is an $S^\infty$ bundle over all but three points of $B$. Since each $S^\infty$ fiber is simply connected, it lifts to $\widetilde{\mathbb B}$. Thus the full preimage of each such fiber in $\widetilde{\mathbb B}$ is a disjoint union of copies of $S^\infty$. Taking the fiberwise-$\pi_0$ of the map $\widetilde{\mathbb B} \to B$ collapses each of these copies of $S^\infty$ to a point, and we see that the resulting map $\fibpizero(\widetilde{\mathbb B} \to B) \to B$ is a covering space away from the three special points of $B$. Considering those points as well, the map is a branched cover. The generic degree of the branched cover is the same as the degree of the cover $\widetilde{\mathbb B} \to \mathbb B$, and this is the quantity we seek.

    So now we consider this branched covering of $B$ branched along 3 points. Considering the local structure over those points, again as provided by \cref{prop: structure of the base space projection}, we see that $\fibpizero(\widetilde{\mathbb B} \to B)$ will have the local structure of a branched cover of the disk with branching order dividing 2, 3, or $d$ (according to the point in question).
    
    We construct a branched cover $X \to B$ and will see that $X$ is our $\fibpizero(\widetilde{\mathbb B} \to B)$. So split $B \approx \mathbb{CP}^1$ along a pair of embedded arcs disjoint on their interiors, one arc from $[0:1]$ to $[27:-4]$ and the other from $[27:-4]$ to $[1:0]$. The resulting topological quadrilateral is simply connected and has no branch points in its interior. Create a surface $X$ by gluing edges of this quadrilateral edge-to-edge with 2 copies around preimages of $[0:1]$, 3 copies around preimages of $[1:0]$, and $2d$ copies around preimages of $[27:-4]$. See \cref{fig:branched_cover}. Now $X$ is a surface with a branched covering to $B$ of the appropriate local branching orders. The quadrilateral tiles fit together into a tiling of $X$ by ``equilateral'' triangles that meet $d$ at each vertex. Each such triangle is tiled by three copies of the quadrilateral. 

    \begin{figure}
        \centering
        {\includegraphics[page=2,
            viewport={1.0in 2.0in 5.0in 4.6in},
            clip, scale=0.9]{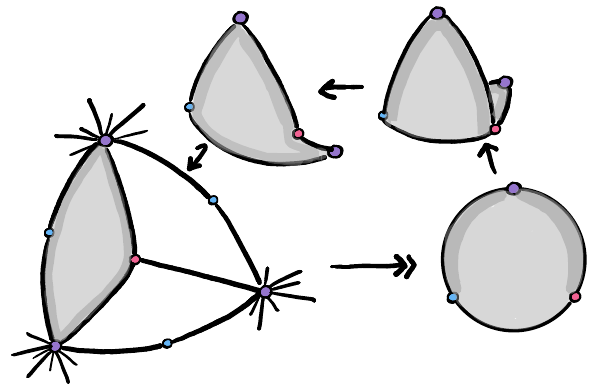}}
        \caption[Branched cover of the triangle orbifold]{The branched cover of $B$ is constructed by splitting along two arcs and arranging copies of the resulting figure. The lifts form into ``equilateral'' triangles with $d$ at each vertex.}
        \label{fig:branched_cover}
    \end{figure}

    The surface $X$ is a sphere when $d \leq 5$ and is contractible when $d \geq 6$. In any case, $X$ is simply-connected. The pullback of $\mathbb B \to B$ along $X \to B$ is an $S^\infty$ bundle over $X$. This space is simply connected, i.e. it is the universal cover of $\mathbb B$. This shows that $X = \fibpizero(\widetilde{\mathbb B} \to B)$.

    Finally to understand the degree of the coverings here, note that we have described a regular tiling of $X$ by triangles with $d$ at each vertex. There are $f(3,d)$-many of these triangles by definition. Each triangle is tiled by 3 copies of the lifted quadrilateral, each of which projects injectively to $B$ on its interior. Thus the branched-cover $X \to B$ is $3\cdot f(3,d)$-to-$1$ away from the branch set. This is the same as the degree of the cover $\widetilde{\mathbb B} \to \mathbb B$, giving our result.
    

\end{proof}

\begin{remark}
        Geometrically, one can take $X$ in the above proof to be the round sphere ($d \geq 5$), the Euclidean plane ($d=6$), or the hyperbolic plane ($d \geq 7$). The group of deck transformations of the branched-cover can be made to act by isometries on $X$, giving an action of the group on the principal $S^1$-bundle that is the unit tangent bundle of $UT(X)$. Then the group acts on the the associated $S^\infty$ bundle over $X$, which we denote $\mathbb X$. The structure group of these $S^\infty$ bundles lies in the unitary group $U(\ell^2)$, which is contractible by Kuiper's theorem \cite{Kuiper1965TheSpace}. Thus these $S^\infty$ bundles are trivial, as is the $S^\infty$ bundle $\widetilde{\mathbb B} \to X$. In this way we gain the perspective $B \approx X/\Delta(2,3,d)$ and ${\mathbb B} \approx \mathbb X/\Delta(2,3,d)$.
\end{remark}

To finish our analysis, we need a quick lemma from low-dimensional topology.

\begin{lemma}\label{lemma: linking numbers}
    Any two regular orbits of the action $S^1 \curvearrowright S^3$ have linking number 6 (viewed as oriented knots in the 3-sphere).
\end{lemma}
\begin{proof}
    There are many ways to compute this. In \cref{fig:linked_trefoils} we show two regular orbits sitting on a common standard torus in $S^3$. Considering this as an oriented link diagram, we can count the crossings of the two components (with sign) and divide by 2 to get the linking number of the two knots (see e.g. \cite[\textsection 5.D]{Rolfsen2003KnotsLinks}).

    \begin{figure}
        \centering
        {\includegraphics[page=2,
            viewport={0.8in 8.6in 2.8in 10.6in},
            clip]{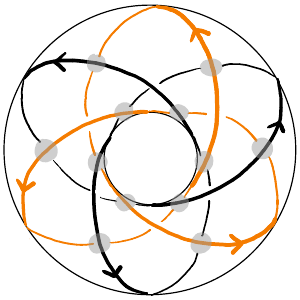}}
        \caption[Two linked trefoils]{Two oriented trefoil knots on an unknotted torus with the 12 component-wise crossings marked.}
        \label{fig:linked_trefoils}
    \end{figure}

    We see that there are 12 crossings, all positively oriented. This gives the lemma.
\end{proof}

\begin{proposition}\label{prop: boundary map}
    When $d \leq 5$, we have $\pi_2(\mathbb B) \approx \mathbb Z$. Otherwise, $\pi_2(\mathbb B)$ is trivial. In either case, the boundary map $\partial:\pi_2(\mathbb B) \to \pi_1(S^1)$ has image $\frac{f(3,d)}{2}\mathbb Z$ in $\pi_1(S^1) \approx \mathbb Z$ (where we interpret this as the trivial subgroup when $f(3,d)$ is infinite).
\end{proposition}
\begin{proof}
    Since higher homotopy groups are preserved by covers, we have that $\pi_2(\mathbb B)$ is identical to $\pi_2$ of the universal cover $\widetilde{\mathbb B}$. But as in the proof of \cref{prop: order of triangle group}, the space $\widetilde{\mathbb B}$ is an $S^\infty$-bundle over the model geometry $X$. Since $S^\infty$ is contractible, this gives $\pi_2(\mathbb B) \approx \pi_2(\widetilde{\mathbb B}) \approx \pi_2(X)$. This group is $\mathbb Z$ when $X$ is spherical (i.e. when $d \leq 5$) and is trivial when $X$ is contractible (i.e. when $d \geq 6$). For the rest of the proof, suppose $d \leq 5$ so that we are in the nontrivial case with $X \approx S^2$. The reader should have the construction of \cref{prop: order of triangle group} in mind for the rest of this proof.

    Now recall how the boundary map in a long exact sequence works. An element of $\pi_2(\mathbb B)$ is represented by a map $\delta:(D^2,\partial D^2) \to (\mathbb B, \text{pt})$. The homotopy lifting property of fiber bundles lets us lift $\delta$ to a map $\tilde \delta:(D^2,\partial D^2) \to (\mathbb E, F)$ where $F$ is the circle fiber over the chosen basepoint of $\mathbb B$. The boundary of this lifted map is a loop $\partial D^2 \to F$, and then the induced element in the fundamental group is the image of the element $\delta \in \pi_2(\mathbb B)$ in the fiber group $\pi_1(F)$.

    In our case, $\pi_2(\mathbb B) \approx \pi_2(\widetilde{\mathbb B})$ is infinite cyclic, generated by a map $(D^2,\partial D^2) \to \widetilde{\mathbb B}$ that projects to $\delta:D^2 \to \mathbb B$ which then lifts to $\tilde \delta:D^2 \to \mathbb E$. We arrange things so that $D^2 \to \tilde{\mathbb B} \xrightarrow{\sim} S^2$ is smooth, 1-to-1 on its interior, and maps the boundary of the disk to a point not fixed by the triangle group action. Then $D^2 \xrightarrow{\delta} \mathbb{B} \xrightarrow{p_B} B$ is smooth and generically $3f(3,d)$-to-1 by the constructions of \cref{prop: order of triangle group}. We'll conflate notation now and write $\delta$ for $p_B \circ \delta$ and $\tilde \delta$ for $p_E \circ \tilde \delta$.

    \begin{figure}
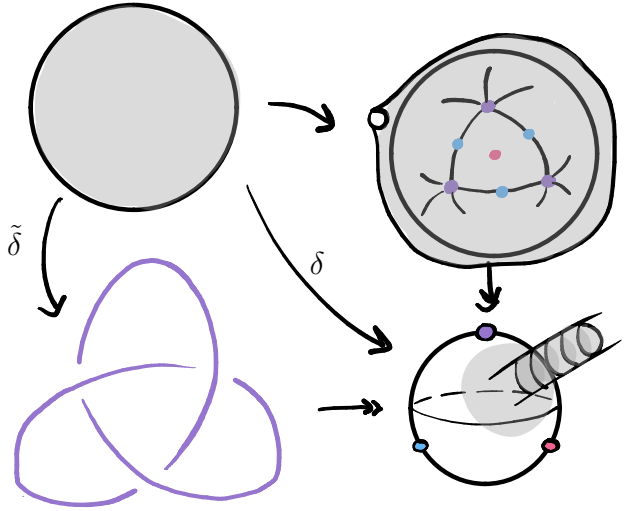

        \centering

        \begin{overpic}[
            page=2,
            viewport={4.1in 7.6in 7.4in 10.4in},
            clip, grid=false
        ]{figures.pdf}
            \put(2,43){$\tilde \delta$}
            \put(50,40){$\delta$}
        \end{overpic}
        \caption[Boundary map]{A generator $\delta$ of $\pi_2(\mathbb B)$ lifts to a map $\tilde \delta:D^2 \to \mathbb E$ with boundary along a fiber. The map $\delta$ is generically ($3 \cdot f(3,d)$)-to-1.}
        \label{fig:boundary_map}
    \end{figure}

    So then $\partial \tilde \delta$ is a smooth immersed curve in the 3-sphere. It is some multiple $k$ of a generic circle fiber $F$ of the Seifert fibration, and we can compute its linking number with another generic fiber $F'$. On the one hand, 
    \begin{align*}
        \lk(\partial \tilde \delta,F') = k \cdot \lk(F,F') = 6k
    \end{align*}
    by \cref{lemma: linking numbers}. On the other hand, this curve bounds a disk and so we can consider algebraic intersection number $\hat i$. This gives
    \begin{equation*}
        \lk(\partial \tilde \delta,F') = \hat i(\tilde \delta,F') = \hat i(\delta,\text{pt}') = \ 3 \cdot f(3,d).
    \end{equation*}
    The second equality comes from projecting to $B$ and the third is just the degree of $\delta$ as above. Comparing the results we see that $k$, the degree of the map $\partial \tilde \delta \to F$, is equal to $f(3,d)/2$. This completes the proof.
\end{proof}

Now in the case when $d \leq 5$, we can relabel the pieces of the exact sequence \eqref{eqn: long exact sequence} as follows:

\begin{adjustbox}{width=1.1\textwidth,center}
\begin{tikzcd}
\dotsb \arrow[r] & \pi_2(\mathbb B) \arrow[r, "\partial"] \arrow[d, "\approx"] & \pi_1(S^1) \arrow[r] \arrow[d, "\approx"] & \pi_1(\mathbb E) \arrow[r] \arrow[d, "\approx"] & \pi_1(\mathbb B) \arrow[r] \arrow[d, "\approx"] & \pi_0(S^1) \arrow[r] \arrow[d, "\approx"] & \dotsb \\
                 & \mathbb Z \arrow[r, "{\times \frac{f(3,d)}{2}}"]            & \mathbb Z \arrow[r, "1 \mapsto \Delta^2"]          & B_3(d) \arrow[r]                                     & {\Delta(2,3,d)} \arrow[r]                       & \{1\}                                     &       
\end{tikzcd}
\end{adjustbox}

In the case $d \geq 6$, the lower left term is the trivial group. This gives \cref{thm: order of center} as well as Coxter's formula \eqref{eq: Coxeter formula} in the $n=3$ case:

\begin{corollary}\label{cor: Coxeter's formula in n=3}
    The order of the truncated braid group $B_3(d)$ is
    \begin{equation*}
        \frac{f(3,d)}{2} \cdot 3f(3,d) = \left( \frac{f(3,d)}2 \right)^2 3!.
    \end{equation*}
\end{corollary}

\section{Final remarks}\label{sec: final remarks}
Of course, it would be ideal to give a topological interpretation of the connection between \textit{any} truncated braid group and its associated regular tiling from \cref{thm: Coxeter formula}. Parts of our story work just the same:
\begin{itemize}
    \item The braid group $B_n$ is the fundamental group of the complement of a discriminant variety in $\mathbb C^{n-1}$. This complement deformation retracts to the complement of a knotted subspace of $S^{2n-3}$.
    \item There is an $S^1$ action on $S^{2n-3}$ that preserves the discriminant locus. Generic points have trivial stabilizer and trace out a loop that represents the full twist $\Delta^2 \in B_n$.
    \item The quotient space for the group action $S^{2n-3}/S^1$ is topologically $\mathbb{CP}^{n-2}$, but the projection has singularities for the exceptional orbits. Some subspace is removed for the discriminant locus.
\end{itemize}
One would want to ``sew in'' an order-$d$ orbifold set along the discriminant locus just as we did along the trefoil knot, but the higher dimensional story is more complicated. For instance, the discriminant locus in $S^5$ for the $n=4$ case is an immersed $S^3$ with self-intersection. The self-intersections correspond to collisions of more than two points in the configuration space of points.

Supposing that one could circumvent these difficulties and fill in the discriminant locus, we hope that the orbifold version of $\mathbb{CP}^{n-2}$ would carry a geometric structure. In the finite cases, the orbifold covering space might correspond to so-called ``regular complex polytopes'' as described in \cite{Coxeter1991RegularPolytopes}.

As a final thought, the truncated braid groups are examples of what are known as \emph{complex reflection groups} \cite{Lehrer2009UnitaryGroups}. These are like real reflection groups, i.e. Coxeter groups, except that a ``complex reflection'' can have any finite order, not just order 2. Perhaps the techniques in this article, and the extensions proposed in this section, could be applied to understand complex reflections groups more generally. The first author wonders if there could be some application to understanding Artin groups as well.

\printbibliography

\end{document}